\documentclass[12pt,reqno]{amsart}

\usepackage{amsfonts,amsmath,amssymb,amsthm}
\usepackage{enumitem}
\usepackage{amsaddr}
\usepackage{mathrsfs}
\usepackage{mathabx}
\usepackage{bbm}
\usepackage{centernot}
\usepackage{mathtools}
\usepackage{graphicx}

\newtheorem{Theorem}{Theorem}[section]

\newtheorem{Proposition}[Theorem]{Proposition}

\newtheorem{Definition}[Theorem]{Definition}
		
\numberwithin{equation}{section}
\usepackage[a4paper,margin=3cm]{geometry}
\begin{document}
\sloppy
\title[Riemann integrability in normed spaces]
{On the Riemann integrability of the norm of a path in normed spaces}

\author[B. \'Alvarez-Samaniego]{Borys \'Alvarez-Samaniego}
\address{\vspace{-8mm}N\'ucleo de Investigadores Cient\'{\i}ficos\\
	    Facultad de Ciencias\\
	    Universidad Central del Ecuador (UCE)\\
    	Quito, Ecuador}
\email{balvarez@uce.edu.ec, borys\_yamil@yahoo.com}

\author[W. P. \'Alvarez-Samaniego]{Wilson P. \'Alvarez-Samaniego}
\address{\vspace{-8mm}N\'ucleo de Investigadores Cient\'{\i}ficos\\
        Facultad de Ciencias\\
	    Universidad Central del Ecuador (UCE)\\
    	Quito, Ecuador}
\email{wpalvarez@uce.edu.ec, alvarezwilson@hotmail.com}

\author[L. Rivera]{Luis Rivera}
\address{\vspace{-8mm}N\'ucleo de Investigadores Cient\'{\i}ficos\\
        Facultad de Ciencias\\
    	Universidad Central del Ecuador (UCE)\\
    	Quito, Ecuador}
\email{ldriverab@uce.edu.ec}

\date{July 11, 2023.} 
%\thanks{Data sharing not applicable to this article as no datasets were 
%generated or analysed during the current study.}

\begin{abstract}
A useful result is that if a bounded complex-valued path is Riemann-integrable, 
then its modulus is also Riemann-integrable. The extension of this last result to 
bounded paths taking values in a normed space is affirmed, as being true, in 
\cite{kolmogorov}. However, we show that if a bounded path taking values in a normed 
space is barely Riemann-integrable, then it is not really guaranteed that the norm 
of this path is also Riemann-integrable.
\end{abstract}


\maketitle
%%%%%%%%%%%%%%%%%%%%%%%%%%%%%%%%%%%%%%%%%%%%%%%%%%%%%%%%%%%%%%%%%%%%%%%%%%%%%%%%%%%%%%%%
%%%%%%%%%%%%%%%%%%%%%%%%%%%%%%%%%%%%%%%%%%%%%%%%%%%%%%%%%%%%%%%%%%%%%%%%%%%%%%%%%%%%%%%%
For notational convenience, we write $\mathbb{Q}$ to represent the set 
of rational numbers. The symbols $\mathbb{R}$, and $\mathbb{C}$, 
are used to denote the sets of real and complex numbers, respectively, 
$\mathbb{N}:=\{0,1,2, \ldots\}$ is the set of all natural numbers. 
The cardinality of the set $C$ is denoted by $\#C$. Furthermore, if 
$\emptyset\neq A\subset X$, we write $\chi_{A}:X\longrightarrow \mathbb{R}$  
to denote the \emph{characteristic function} of the subset $A$, which maps
every element of $A$ to $1$, and all other element to $0$. 

We recall (see \cite{castro}) that a \textit{partition}, $\mathcal{P}$, of a 
compact interval $[a,b]\subset \mathbb{R}$, is a nonempty finite collection 
$\mathcal{P} := \{I_{1},\ldots, I_{\#\mathcal{P}}\}$ of mutually disjoint intervals
such that $I_1 := [a, b]$, if $\#\mathcal{P} = 1$, and $I_1 := [x_0, x_1)$, 
$\ldots$, $I_{\#\mathcal{P}} := [x_{\#\mathcal{P}-1}, x_{\#\mathcal{P}}]$, 
where $x_0:=a$, $x_{\#\mathcal{P}}:=b$, and $x_0 \le \ldots \le x_{\#\mathcal{P}}$, 
if $\#\mathcal{P} \in \{ 2, 3, \ldots\}$. The \textit{diameter} of a partition,  
$\mathcal{Q}$, of a compact interval, denoted by $\text{diam}(\mathcal{Q})$, is the 
maximum of the lengths of the elements of $\mathcal{Q}$. Next, we give a classical 
definition.
%%%%%%%%%%%%%%%%%%%%%%%%%%%%%%%%%%%%%%%%%%%%%%%%%%%%%%%%%%%%%%%%%%%%%%%%%%%%%%%%%%%%%%%
%%%%%%%%%%%%%%%%%%%%%%%%%%%%%%%%%%%%%%%%%%%%%%%%%%%%%%%%%%%%%%%%%%%%%%%%%%%%%%%%%%%%%%%
\begin{Definition}[Riemann integral, \cite{castro}]
    Let $I:=[a,b]\subset \mathbb{R}$ be a compact interval, and let 
    $f:I\longrightarrow E$ be a bounded path, i.e., a bounded function, 
    where $(E,|\cdot|)$ is a normed space over the real or complex numbers. 
    If there exists $A\in E$ such that for all $\epsilon>0$, there is a $\delta>0$
    such that for any partition $\mathcal{P}$ of $I$, with 
    $\mathcal{P}:=\{I_{1},\ldots,I_{\#\mathcal{P}}\}$, and for every finite set 
    $\{x_{1},\ldots,x_{\#\mathcal{P}}\}$ such that for all 
    $j\in \{1,\ldots,\#\mathcal{P}\}$, $x_{j}\in I_{j}$, 
    \[
        0\leq \text{diam}(\mathcal{P})<\delta 
        \Rightarrow 
        \left|\sum_{j=1}^{\#\mathcal{P}}f(x_{j})\cdot \text{vol}(I_{j})-A\right|
        <\epsilon,
    \]
    where $\text{vol}$ is the length of the interval under consideration, we say 
    that the function $f:I\longrightarrow E$ is \textit{Riemann-integrable} and 
    \[
        \int_{a}^{b}f(x)dx = A\in E. 
    \]
\end{Definition}
%%%%%%%%%%%%%%%%%%%%%%%%%%%%%%%%%%%%%%%%%%%%%%%%%%%%%%%%%%%%%%%%%%%%%%%%%%%%%%%%%%%%%%%
%%%%%%%%%%%%%%%%%%%%%%%%%%%%%%%%%%%%%%%%%%%%%%%%%%%%%%%%%%%%%%%%%%%%%%%%%%%%%%%%%%%%%%%
We now proceed to prove the main result of this manuscript.
\begin{Proposition}
    There exist a normed space, $(E,|\cdot|)$, and a bounded path, 
    $f:[0,1]\longrightarrow E$, such that $f:[0,1]\longrightarrow E$ 
    is Riemann-integrable, but $|f|:[0,1]\longrightarrow [0,+\infty)$ 
    is not Riemann-integrable.
\end{Proposition}
\begin{proof}
We consider the Banach space 
$(\ell^{\infty} (\mathbb{N}),|\cdot|_{\infty})$, see \cite{iorio}, where 
\begin{equation*}
 \ell^{\infty} (\mathbb{N}) := \{ \alpha := (\alpha_n)_{n \in \mathbb{N}} \in 
 {\mathbb{C}}^{\mathbb{N}} ;  | \alpha |_{\infty} 
 := \sup_{n \in \mathbb{N}} |\alpha_n| < + \infty \}.
\end{equation*}
Let $\mathbb{Q}=\{r_{0},r_{1},\ldots\} \subset \mathbb{R}$ be an enumeration of 
the set of rational numbers, where for all $i, j\in \mathbb{N}$ such that 
$i \neq j$, $r_{i}\neq r_{j}$. We now take the function
\begin{eqnarray*}
    F:[0,1] &\longrightarrow& \ell^{\infty}(\mathbb{N})\\
    t &\longmapsto& F(t) := 
    \begin{cases}
        (\chi_{\{k\}}(n))_{n\in \mathbb{N}} &\text{, if } t=r_{k}, \text{ with } 
        k\in \mathbb{N},\\
        (0)_{n\in \mathbb{N}} &\text{, if } t\in [0,1]\setminus \mathbb{Q}. 
    \end{cases}
\end{eqnarray*}
\begin{itemize}
    \item It is immediate to see that 
    $F:[0,1]\longrightarrow \ell^{\infty}(\mathbb{N})$ is well-defined. 
    \item Let $t\in [0,1]$. If $t\in \mathbb{Q}$, then
    there exists $m:=m_{t}\in \mathbb{N}$ such that $t=r_{m}$, and
    $|F(t)|_{\infty}:=|(\chi_{\{m\}}(n))_{n\in \mathbb{N}}|_{\infty}
    =1$.
    On the other hand, if $t\in [0,1]\setminus \mathbb{Q}$, we see that
    $|F(t)|_{\infty}:=|(0)_{n\in \mathbb{N}}|_{\infty} = 0$.
    Thus, $F$ is a bounded path.
    \item Let $\epsilon>0$. Let $\mathcal{P}$ be a partition of $[0,1]$, 
    with $\mathcal{P}:=\{I_{1},\ldots,I_{\#\mathcal{P}}\}$, where
    $0\leq \text{diam}(\mathcal{P})<\epsilon$.
    Let $\{x_{1},\ldots,x_{\#\mathcal{P}}\}$ be a finite set such that 
    for all $i\in \{1,\ldots,\#\mathcal{P}\}$, $x_{i}\in I_{i}$. First,
    we suppose that for all $i\in \{1,\ldots,\#\mathcal{P}\}$, 
    $x_{i}\in [0,1]\setminus \mathbb{Q}$. Then, 
    \begin{align*}
            \left|
            \sum_{i=1}^{\#\mathcal{P}}F(x_{i})\cdot \text{vol}(I_{i}) 
            \right|_{\infty}
            :=
            \left|
            \sum_{i=1}^{\#\mathcal{P}}(0)_{n\in \mathbb{N}}\cdot \text{vol}(I_{i}) 
            \right|_{\infty}
            =
            \left|
            (0)_{n\in \mathbb{N}} 
            \right|_{\infty}
            =0<\epsilon.
    \end{align*}
    We now assume that $J:=\{l\in \{1,\ldots,\#\mathcal{P}\}; x_{l}\in \mathbb{Q}\}
        \neq \emptyset$. So, for all $m\in J$, there exists $k_{m}\in \mathbb{N}$
        such that $x_{m}=r_{k_{m}}$. Hence,
        \begin{align*}
            \left|
            \sum_{i=1}^{\#\mathcal{P}}F(x_{i})\cdot \text{vol}(I_{i}) 
            \right|_{\infty}
            &:=
            \left|
            \sum_{i\in J}F(x_{i})\cdot \text{vol}(I_{i}) 
            \right|_{\infty}\\
            &:=
            \left|
            \sum_{i\in J}(\chi_{\{k_{i}\}}(n))_{n\in \mathbb{N}}\cdot \text{vol}(I_{i}) 
            \right|_{\infty}\\
            &=
            \left|
            \left(\sum_{i\in J}\chi_{\{k_{i}\}}(n)\cdot \text{vol}(I_{i})\right)_{
            n\in \mathbb{N}} 
            \right|_{\infty}\\
            &=\max\{\text{vol}(I_{i}); i\in J\}\\
            &\leq\text{diam}(\mathcal{P})\\
            &<\epsilon.
        \end{align*}    
    Consequently, 
    $F$ is Riemann-integrable,
    and 
    \[
        \int_{0}^{1}F(t)dt = (0)_{n\in \mathbb{N}}\in \ell^{\infty}(\mathbb{N}). 
    \]
    \item Finally, since  
    $|F|_{\infty}=\chi_{\mathbb{Q}\,\cap\, [0,1]}:[0,1]\longrightarrow 
    [0,+\infty)$, we see that
    \[
        \hspace{1
        cm}
        \lambda(\{x\in [0,1]; |F|_{\infty}=\chi_{\mathbb{Q}\,\cap\, [0,1]} 
        \text{ is discontinuous at } x\}) = \lambda([0,1]) = 1>0,
    \]
    where $\lambda$ is the Lebesgue measure
    on $\mathbb{R}$. Therefore,
    $|F|_{\infty}$ is not Riemann-integrable. \qedhere
\end{itemize}
\end{proof}

%%%%%%%%%%%%%%%%%%%%%%%%%%%%%%%%%%%%%%%%%%%%%%%%%%%%%%%%%%%%%%%%%%%%%%%%%%%%%%%%%%%%%%%%
%%%%%%%%%%%%%%%%%%%%%%%%%%%%%%%%%%%%%%%%%%%%%%%%%%%%%%%%%%%%%%%%%%%%%%%%%%%%%%%%%%%%%%%%

\end{document}